\documentclass[11 pt,a4paper]{article}
\usepackage[utf8]{inputenc}

\parskip=\medskipamount
\usepackage{graphicx}
\usepackage{bm}
\usepackage{amssymb, amsmath} 
\usepackage{latexsym} 
\usepackage{graphics} 
\usepackage{url} 
\usepackage{enumerate}
\usepackage[all]{xy}
\usepackage{float}
\usepackage{amsrefs}
\usepackage{color}

\usepackage{comment}

\usepackage{amsthm}

\usepackage[font=small,labelfont=bf]{caption}

\newtheorem{theo}{Theorem}[section]

\newtheorem{lemm}[theo]{Lemma}
\newtheorem{prop}[theo]{Proposition}
\newtheorem{exam}[theo]{Example}

\newtheorem{defi}[theo]{Definition}
\newtheorem{rem}[theo]{Remark}

\title{$n$-gon centers and central lines}
\author{Marta Farr\'e Puiggal\'i\footnote{Department of Mathematics, University of Antwerpen, (Belgium)}, Luis Felipe Prieto-Mart\'inez\footnote{Departamento de Matem\'atica Aplicada, ETS Arquitectura, Universidad Polit\'ecnica de Madrid (Spain)}}

\begin{document}

\maketitle

\begin{abstract} In this paper we provide a review of the concept of center of a $n$-gon, generalizing the original idea given by C. Kimberling for triangles. We also generalize the concept of central line for $n$-gons for $n\geq 3$ and  establish its basic properties.

\medskip


\medskip

\end{abstract}

\section{Introduction}

For $n\geq 3$, let us denote by $\mathcal P_n$ the set of all $n$-gons in the plane with their vertices labelled. The elements in $\mathcal P_n$ can be identified with $n$-tuples $(V_1,\ldots,V_n)$ of elements in $\mathbb R^2$. This is the approach taken in \cite{BS}.

Let us denote by $\rho$ and $\sigma$ the  permutations of $\{1,\ldots, n\}$ given by
\begin{equation}\label{eq.rho}\rho(i)=i+1\mod n \qquad \mbox{and}\qquad \sigma(i)=n+2-i \mod n .\end{equation}

\noindent Consider also the  dihedral group of $2n$ elements $D_n$ generated by $\rho$ and $\sigma$. We say that two $n$-gons $(V_1,\ldots,V_n)$ and $(V_1',\ldots, V_n')$ differ in their labelling if there exists an element $\alpha\in D_n$ such that, for every $i=1,\ldots,n$, $V_i'=V_{\alpha(i)}$.

The study of centers of polygons already began with the \emph{classical triangle centers} (centroid, circumcenter, incenter and orthocenter), which were already known by the ancient greeks. Over the years, many other points related to the triangle were studied. This motivated that in the last decade of the past century, C. Kimberling introduced an abstract definition of triangle center and triangle central line, see for example \cites{K.FE, K.CL}. He removed the idea of geometric significance from these objects in order to allow a systematic study and classification. He also created the \emph{Encyclopedia of Triangle Centers} \cite{K.E}, where he listed all known interesting cases (44072 at this moment).

Interesting points, also named centers, are also known for $n$-gons, for $n\geq 4$, as we will explain in a moment. So, after Kimberling's articles, some attempts were made to extend the study of centers to $n$-gons for $n\geq 4$. We highlight some of them:

\begin{itemize}
    \item A definition of center for $n$-simplices was given in \cite{E}. If we had defined $n$-gons to be sets $\{V_1,\ldots, V_n\}$ instead of $n$-tuples $(V_1,\ldots, V_n)$, we could have followed this approach. But it does not take into account the \emph{adjacency structure} between the vertices of a $n$-gon.

    \item In \cites{AS.Q, AM, M}, among others, some centers of plane quadrilaterals and even $n$-gons are studied, but no formal definition of this concept is provided.
    
    \item In the website \cite{VT} we can find another  \emph{encyclopedia} (or list) of $n$-gon centers and lines, specially focusing on the case $n=4$.

\end{itemize}

\noindent Finally in \cite{PS.C} the authors succesfully generalized the definition of triangle (trigon) center function given by Kimberling to $n$-gons for $n\geq 4$.

This paper is devoted to the concepts of $n$-gon center and $n$-gon central line. The main contributions herein are the following. First, in Section \ref{section.centers}, we provide a new formal definition of \emph{center of a $n$-gon}, which is more clear than the one given in \cite{PS.C}. We prove that our definition and the one in \cite{PS.C} are equivalent (Theorem \ref{theo.interpretation}) and, as a consequence, we solve one of the open questions in that article. Then, after revisiting the ideas of central points and lines for trigons in Section \ref{section.lineskimberling}, we study possible generalizations of the concept of central line for $n\geq 4$. We show that some desirable properties for central lines that hold for trigons (see Remark \ref{theo.main3}) do not hold for central lines of $n$-gons, for $n\geq 4$ (see Theorem \ref{theo.main} and Example \ref{exam.rectangulo}). In Section \ref{sect.symmetry} we study the problem that was already mentioned in \cites{AS.Q, E, PS.C} concerning the relationship between regularity and coincidence of centers. We obtain a result in this direction (see Theorem \ref{theo.symmetry}). Finally, in Section \ref{section.app}, we include some illustrative applications of our approach to concrete problems in plane geometry.

\section{Centers of polygons revisited} \label{section.centers}

For this section the main reference is \cite{PS.C}, but the notation, the use of some terms and the approach are slightly different in this paper. In \cite{PS.C} the authors first introduce \emph{$n$-gon center functions} and then define \emph{centers}, which are a geometric interpretation of the first concept. Here this is done the other way around: we define centers as a concept with geometric meaning and then we introduce $n$-gon center functions as a natural consequence. This is more consistent with the concepts that we will  introduce in the other sections, for instance central lines.

From now on, $\mathcal F_n$ will denote a subset of $\mathcal P_n$ which is \emph{closed with respect to similarities and relabellings}, that is, (1) if $P,Q$ are similar and $P\in\mathcal F_n$ then $Q\in\mathcal F_n$ and (2) if $P,Q$ differ only in their labelling and $P\in\mathcal F_n$, then $Q\in\mathcal F_n$.

\begin{defi} \label{defi} 
 A \textbf{center} is, formally, a function $\Phi:\mathcal F_n\to\mathbb R^2$ such that

\begin{itemize}

\item[(1)] It \emph{commutes with similarities}, that is, for every two similar $n$-gons $P,Q$,  the corresponding points $\Phi(P),\Phi(Q)$ are related via the same plane transformation.

\item[(2)] It is \emph{invariant with respect to relabellings}, that is, if two $n$-gons $P,Q$ differ only in their labellings,  $\Phi(P)=\Phi(Q)$.

\end{itemize}

\end{defi}

Sometimes we will commit an abuse of notation, as is done in the case of classical triangle centers. In Elementary Geometry courses the term \emph{incenter} is used for both the function that maps the set of triangles to the set of points in the plane and for the point corresponding to a concrete triangle. In this paper, if we say that a point $X$ in the plane is a center of a concrete $n$-gon $P$ it means that we assume that there exist some $\mathcal F_n$ containing $P$ and a function $\Phi:\mathcal F_n\subset\mathcal P_n\to\mathbb R^2$ which is a center in the previous sense and satisfies $\Phi(P)=X$. Note that not every point $X$ in the plane is, in general, a center of a given $n$-gon, see Section \ref{sect.symmetry}.

Even for some well known centers we need the family $\mathcal F_n$ to be a proper subset of $\mathcal P_n$. For instance, we may need to exclude from $\mathcal P_n$ the \textbf{flat $n$-gons}, that is, those for which all the vertices are collinear. For $n=3$ this is required for the usual interpretation of the \emph{circumcenter}. Sometimes we need to restrict to an even smaller family $\mathcal F_n$. For example, for $n=4$, the \emph{circumcenter} (the center of the circumcircle if it exists) is only defined for $\mathcal F_n$ being the set of cyclic quadrigons.

Let us consider the space $\mathcal T_n$ to be the quotient space of $\mathcal P_n$ with respect to the equivalence relation induced by plane congruences: two elements $P,Q$ are mapped to the same equivalence class if they are congruent and the corresponding congruence maps the vertex $i$ of $P$ to the vertex $i$ of $Q$ for $i=1,\ldots, n$. The elements in $\mathcal T_n$  can be identified with the corresponding \textbf{distance matrix}, that is, the $n\times n$ square matrix $[d_{ij}]_{1\leq i,j\leq n}$ such that, for $i,j=1,\ldots, n$,   $d_{ij}=d(V_i,V_j)$ for any representative $(V_1,\ldots, V_n)$ of the equivalence class. Not every $n\times n$ hollow, symmetric and positive matrix corresponds to the distance matrix of a set of $n$ points in the plane. For some information about the characterization of such matrices we refer to \cites{BB, LL, PS.C, S.NDIM}.

\medskip
Let us define the following:

\begin{defi} Let $\mathcal F_n$ be a subset of $\mathcal P_n$ which is closed with respect to similarities and relabellings, and let $\widetilde{\mathcal F}_n$ be the corresponding projection onto the quotient space $\mathcal T_n$.   A \textbf{$n$-gon center function} is a function $g:\widetilde{\mathcal F}_n\to\mathbb R$ satisfying the two following properties:

\begin{itemize}

\item[(i)] (symmetry condition) Let $\sigma$ denote the permutation of the set $\{1,\ldots, n\}$ defined in Equation \eqref{eq.rho}. Then $g([d_{ij}])=g([d_{\sigma(i)\sigma(j)}])$.

\item[(ii)] (homogeneity) There exists some $m\in\mathbb{N}$ such that for every $\lambda\in\mathbb R_{> 0}$, $g([\lambda d_{ij}])=\lambda^m g([d_{ij}])$.

\end{itemize}

\end{defi}

 In \cite{PS.C} we can find some motivation for the following theorem, although one of the two directions is not proved there. 
 This result also fills a gap in \cite{PS.C}: it shows that the answer to Open Question 26 is affirmative (since \emph{implicit centers}, in the sense of that article, is a particular case of the definition of center included herein).

\begin{theo}\label{theo.interpretation}  Suppose that we have a map $\Phi:\mathcal F_n\subset \mathcal P_n\to \mathbb R^2$. Then the following two conditions are equivalent:

\begin{itemize}

\item[(a)] $\Phi$ is a $n$-gon center.

\item[(b)] There exists a $n$-gon center function $g:\widetilde{\mathcal F_n}\to\mathbb R$ (which may not be unique) such that, for every $P=(V_1,\ldots, V_n)$ in $\mathcal F_n$, 
$$\Phi(P)=\lambda_1V_1+\lambda_2 V_2+\ldots+\lambda_n V_n,$$
\noindent where, for $k=1,\ldots, n$,
$$\lambda_k=\frac{g([d_{\rho^{k-1}(i)\rho^{k-1}(j)}])}{g([d_{ij}])+g([d_{\rho(i)\rho(j)}])+\ldots+g([d_{\rho^{n-1}(i)\rho^{n-1}(j)}])} $$

\noindent and $\rho$ is the permutation defined in Equation \eqref{eq.rho}.

\end{itemize}

\end{theo}

\begin{proof} All subindices will be considered modulo $n$ throughout this proof. The implication $(b)\Rightarrow (a)$ follows from a combination of Theorems 7 and 10 in \cite{PS.C}. Let us prove that $(a)\Rightarrow(b)$. To do so, we will determine such a function $g$. Let $P=(V_1,\ldots,V_n)$ be a representative of the equivalence class $[d_{ij}]_{1\leq i,j\leq n}$.

Let us suppose that $P$ is not flat. This implies that there exists at least one value $k$, $1\leq k\leq n$, satisfying

\noindent \textbf{Property (*):} the three consecutive vertices $V_{k-1},V_k,V_{k+1}$  are not collinear.

For every $k$ satisfying the Property (*),  there is a unique triple of real numbers $\mu_{k,k-1},\mu_{kk}, \mu_{k,k+1}$ such that 
\begin{equation} \label{eq.mus} \mu_{k,k-1}+\mu_{kk}+\mu_{k,k+1}=1\text{ and }\Phi(P)=\mu_{k,k-1}V_{k-1}+\mu_{kk}V_k+\mu_{k,k+1}V_{k+1}\end{equation} 

Let us set $\mu_{km}=0$ for $1\leq m\leq n$, $m\neq k-1,k,k+1$ so that
$$\Phi(P)=\mu_{k1}V_1+\ldots+\mu_{kn}V_n .$$

\noindent For every $k$ not satisfying the Property (*), let us set $\mu_{km}=0$, for $m=1,\ldots, n$. Suppose that there exist $r$ subindices satisfying the Property (*). The numbers $\mu_{k1}$ can be viewed as functions of the vertices $V_1,\ldots,V_n$. We define 
$$g(V_1,\ldots,V_n)=\frac{1}{r}\sum_{k=1}^n\mu_{k1}.$$

First, notice that $g$ assigns the same value to every element in the same equivalence class in $\mathcal T_n$, since, for every plane congruence $T$, the coefficients $\mu_{k,k-1},\mu_{kk},\mu_{k,k+1}$ appearing in Equation \eqref{eq.mus} and the ones satisfying
$$\mu_{k,k-1}+\mu_{kk}+\mu_{k,k+1}=1\text{ and }\Phi(T(P))=\mu_{k,k-1}T(V_{k-1})+\mu_{kk}T(V_k)+\mu_{k,k+1}T(V_{k+1})$$
coincide. So, with an abuse of notation, we will write  $g(V_1,\ldots,V_n)=g([d_{ij}])$.

Second, note that
$$\Phi(P)=\left(\frac{1}{r}\sum_{k=1}^n\mu_{k1}\right)V_1+\ldots+\left(\frac{1}{r}\sum_{k=1}^n\mu_{kn}\right)V_n $$

\noindent and the sum of the coefficients of this affine combination of the vertices equals 1. Moreover,  for $m=2,\ldots, n$,  $(\frac{1}{r}\sum_{k=1}^n\mu_{km})=g([d_{\rho^{m-1}(i)\rho^{m-1}(j)}])$.

Third, note that this function $g$ is a $n$-gon center function. It is easy to see that this function satisfies the symmetry condition. To prove the homogeneity condition it suffices to see that, as functions of the distances in $[d_{ij}]$, for every $\lambda\in\mathbb R_{>0}$, $\mu_{km}([\lambda d_{ij}])=\mu_{km}([d_{ij}])$. So $g([\lambda d_{ij}])=g([ d_{ij}])$.

Finally, we need to consider the special case in which $P$ is flat. 

\begin{itemize} 

\item If all the vertices coincide, then we can easily check that $\Phi(P)$ must coincide also with the vertices. So the statement is trivial: we can take $g([d_{ij}])=1$.

\item If all the vertices are collinear but do not coincide in the same point, then $\Phi(P)$ is also collinear to them. 

Property (*) should be replaced by: the three consecutive vertices $V_{k-1}$,$V_k$ and $V_{k+1}$ do not coincide. We have guaranteed that there exists at least one value $k$, $1\leq k\leq n$ satisfying this new property.

Now, in the paragraph under Property (*), the numbers $\mu_{k,k-1},\mu_{kk},\mu_{k,k+1}$ are not in general unique. But they are unique if we, furthermore, impose the extra condition of $\mu_{k,k-1}=\mu_{k,k+1}$. The rest of the proof remains as in the non-flat case.

\end{itemize}

\end{proof}

We call the map $\varphi:\widetilde{\mathcal F_n}\to \mathbb{R}^n$, $\varphi([d_{ij}])= (\lambda_1([d_{ij}]),\ldots,\lambda_n([d_{ij}]))$,
a \textbf{coordinate map} (it is not unique in general). Here $[d_{ij}]$ is the distance matrix corresponding to $P$ and $(\lambda_1(d_{ij}),\ldots,\lambda_n(d_{ij}))$ is the corresponding set of coefficients of the affine combination from statement (b) in the previous theorem. Let us recall that the tuples $$(\lambda_1([d_{ij}]),\ldots,\lambda_n([d_{ij}]))$$ are not strictly coordinates for $n>3$ since two different tuples may correspond to the same point.

Some examples of centers appear in Section 5 of \cite{PS.C} (together with their $n$-gon center functions), some more in \cite{M} (for $n=4$ with a similar notation to the one used here) and many more  in the website \cite{VT} (in particular for $n=4$ but without the information about the center function).

For completeness we include two examples. The first one is a particularly simple but artificial example which is theoretically useful. The second one, the diagonal crosspoint of a quadrigon, only makes sense for the case $n=4$. We have found a surprising lack of satisfactory formulas describing the diagonal crosspoint of a quadrigon in terms of its elements, so we hope that the one presented here is of its own interest.

\begin{figure}[h!]
\includegraphics[width=0.9\textwidth]{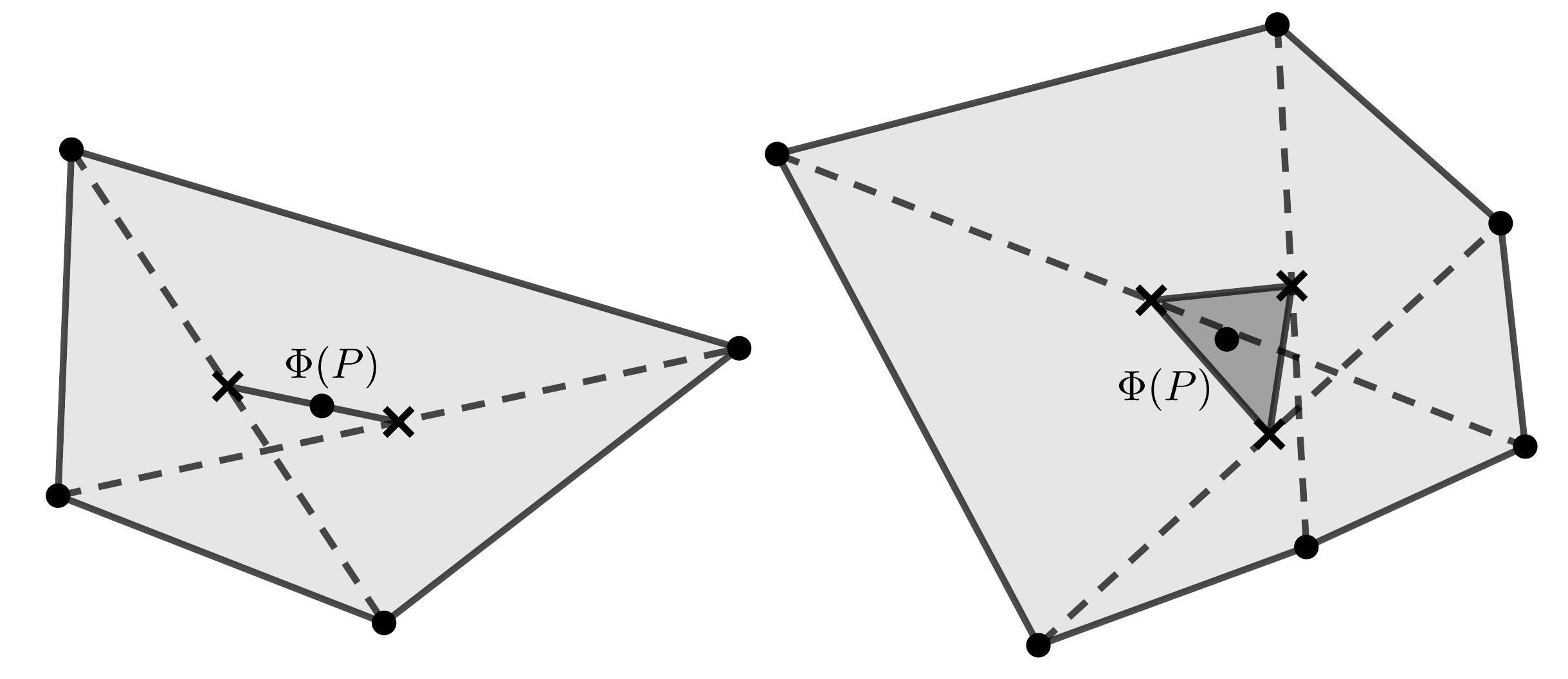}
\label{figure.simplecenter}
\caption{Here we can see the simple center $\Phi(P)$ for a tetragon (on the left) and for an hexagon (on the right). The midpoints of each diagonal are marked with the symbol $\times$.}
\end{figure}

\begin{exam} Let $n$ be even. We define the \textbf{simple center} of a $n$-gon to be
$$\Phi_S(P)=\frac{1}{d_{1,\frac{n}{2}+1}+\ldots+d_{n,\frac{n}{2}}}(d_{1,\frac{n}{2}+1}V_1+\ldots+d_{n,\frac{n}{2}}V_n) \, . $$

\noindent This point is an affine combination of the midpoints of the diagonals between opposite points in $P$, where the coefficients or this new affine combination are proportional to the lengths of the corresponding diagonals (see Figure 1).

\end{exam}

\begin{exam} Let $\mathcal F_4$ be the set of convex quadrigons with no three of their vertices collinear. The function $\Phi:\mathcal F_4\to\mathbb R^2$ given by
$$\Phi(P)=\frac{g(d_{ij})V_1+g(d_{\rho(i)\rho(j)})V_2+g(d_{\rho^2(i)\rho^2(j)})V_3+g(d_{\rho^3(i)\rho^3(j)})V_4}{g(d_{ij})+g(d_{\rho(i)\rho(j)})+g(d_{\rho^2(i)\rho^2(j)})+g(d_{\rho^3(i)\rho^3(j)})},$$
\noindent where
$$g([d_{ij}])=\sqrt{4d_{34}^2d_{24}^2-\left(d_{34}^2+d_{24}^2-d_{23}^2\right)^2}+\sqrt{4d_{23}^2d_{24}^2-\left(d_{23}^2+d_{24}^2-d_{34}^2\right)^2},$$
is a tetragon center function that maps each tetragon $P=(V_1,\ldots, V_n)$ to the crosspoint of its diagonals.

Let us denote by $a,b,c,d$ the distances between $\Phi(P)$ and each of the vertices $V_1,V_2,V_3,V_4$ respectively, and by $\alpha$, $\beta$ the angles $\angle(V_3,V_4,\Phi(P))$ and $\angle(V_4,\Phi(P),V_3)$ as shown in the following diagram.

\begin{center}
\includegraphics[width=7cm]{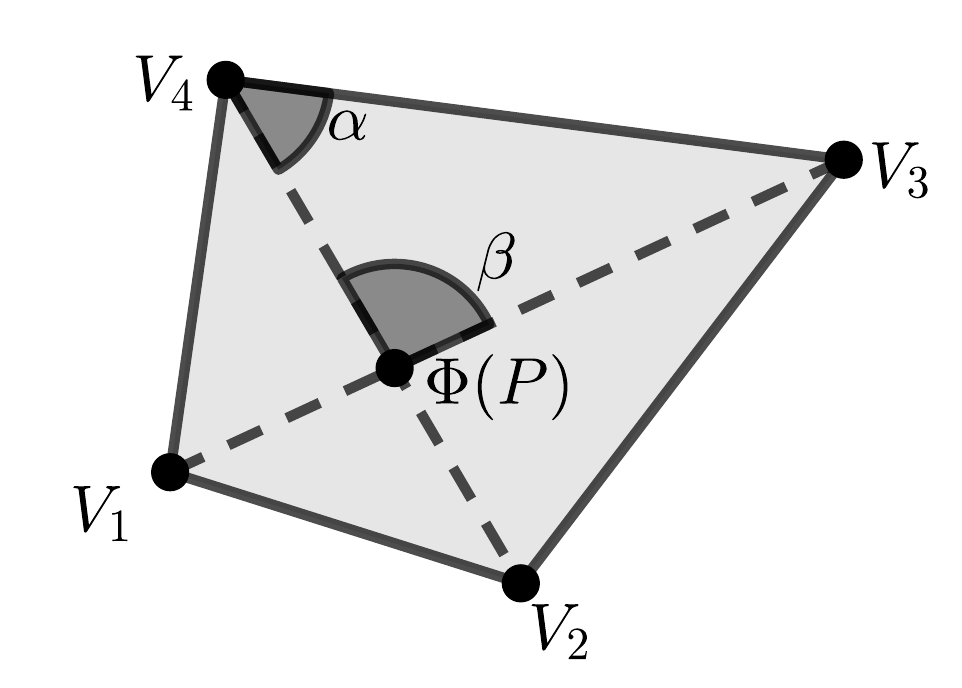}
\end{center}

Using the \emph{Law of Sines} and the \emph{Law of Cosines} in the trigons $(V_3,V_4,\Phi(P))$ and $(V_3,V_4,V_2)$ respectively we get
\begin{eqnarray*}
c&=&\frac{\sin(\alpha)}{\sin(\beta)}d_{34} =\frac{\sqrt{1-\cos^2(\alpha)}}{\sin(\beta)}d_{34}=\frac{\sqrt{1-\left(\frac{d_{34}^2+d_{24}^2-d_{23}^2}{2d_{34}d_{24}}\right)^2}}{\sin(\beta)}d_{34}\\
&=&\frac{\sqrt{4d_{34}^2d_{24}^2-\left(d_{34}^2+d_{24}^2-d_{23}^2\right)^2}}{2d_{24}\sin(\beta)} .
\end{eqnarray*}

But this expression is not suitably symmetric in the subindices $2$ and $4$, so we need to repeat the same argument for the trigons $(V_3,V_2,\Phi(P))$ and $(V_3,V_2,V_4)$ and then take the average to get
$$
c=\frac{\sqrt{4d_{34}^2d_{24}^2-\left(d_{34}^2+d_{24}^2-d_{23}^2\right)^2}+\sqrt{4d_{23}^2d_{24}^2-\left(d_{23}^2+d_{24}^2-d_{34}^2\right)^2}}{4d_{24}\sin(\beta)} .
$$

We can obtain analogous formulas for $a,b$ and $d$. Finally note that
$$\Phi(P)=\frac{c}{d_{13}}V_1+\frac{a}{d_{13}}V_3=\frac{d}{d_{24}}V_2+\frac{b}{d_{24}}V_4 ,$$
\noindent so the desired formula follows from taking the average
$$\Phi(P)=\frac{1}{2}\left(\frac{c}{d_{13}}V_1+\frac{d}{d_{24}}V_2+\frac{a}{d_{13}}V_3+\frac{b}{d_{24}}V_4\right) .$$

\end{exam}

To understand how the centers of a fixed $n$-gon are distributed in the plane, we need the following result:

\begin{lemm} \label{lemma.d} Let $\Phi_A,\Phi_B,\Phi_C$ be three $n$-gon centers defined for the same family $\mathcal F_n$. Let $\mu_A,\mu_B,\mu_C\in\mathbb R$ such that $\mu_A+\mu_B+\mu_C=1$. Then the function $\Phi:\mathcal F_n\to\mathbb R^2$ defined as the affine combination
$$\Phi=\mu_A\Phi_A+\mu_B\Phi_B+\mu_C\Phi_C $$

\noindent is also a $n$-gon center.

\end{lemm}

\begin{proof} $\Phi$ obviously commutes with similarities and with relabellings.
\end{proof}

As a consequence, if a given $n$-gon  has two non-coincident centers, then every point in the line passing through them is a center. This is not a contradiction with the statement in \cite[Section 5]{K.FE}. Recall that, for us, given a $n$-gon, a point $X$ in the plane is a center if there exists a center $\Phi$ such that $\Phi(P)=X$. In turn, if a given $n$-gon has three affinely independent centers, then every point in the plane is a center. This reasoning leads to the following:

\medskip

\noindent \textbf{Question:} \emph{When does a $n$-gon have two centers which are not coincident? When does a $n$-gon have three centers which are affinely independent?}

\medskip

This question will be discussed in Section \ref{sect.symmetry}, but the answer is ``in most cases''.

Finally, we should keep in mind that we have provided a theoretical concept of center, but not every center is interesting. We are mostly interested in centers with geometric meaning.

\section{Review of central points and lines for triangles} \label{section.lineskimberling}

The original definition of center function provided by Kimberling only applied for triangles. In his  approach,  the center is recovered from the center function using trilinear coordinates. In contrast, the approach in the previous section and in \cite{PS.C} (applied to the case $n=3$) uses barycentric coordinates. In \cite{PS.C} there is a discussion on how to change from one point of view to the other, but we also include a brief summary here.

For $n=3$, the elements in $\mathcal T_3$ can be identified with triples $(a,b,c)$ corresponding to the sidelengths.
If $[t_1:t_2:t_3]$ are the trilinear coordinates of a point in a trigon with sidelengths $(a,b,c)$ then the corresponding barycentric coordinates are $[at_1:bt_2:ct_3]$. So Kimberling's definition of triangle center function coincides with our definition of $n$-gon center function for $n=3$, but the geometric interpretation (that is, the center $\Phi$ corresponding to a given center function)  is not the same. For Kimberling, the point corresponding to a center function $g$ and a trigon $P$ is the point $X$ with trilinear coordinates $[g(a,b,c):g(b,c,a):g(c,a,b)]$ with respect to $P$ and for us it is a different point $Y$ with barycentric coordinates $[g(a,b,c):g(b,c,a):g(c,a,b)]$ with respect to the vertices of $P$. For the rest of this section (and only for this section), for a given center function $g$, we will use the notation $\Phi^K:\mathcal F_3\subset \mathcal P_3\to \mathcal R^2$ for the function that maps each triangle to its center, according to Kimberling's interpretation.

The concept of central line was introduced in Kimberling's paper \cite{K.CL}, which we use as the main reference for this section. Kimberling defines a \textbf{triangle central line} as a line passing through two triangle centers. Let $\mathcal L$ denote the set of lines in $\mathbb R^2$.
\begin{defi}
A triangle central line is a function $\Psi^K:\mathcal F_3\to \mathcal L$ such that there are two centers (in the sense of Kimberling) $\Phi_1^K,\Phi_2^K$ satisfying that for every element $P\in\mathcal F_3$, $\Psi^K(P)$ passes through $\Phi_1^K(P)$ and $\Phi_2^K(P)$.
\end{defi}

Note that $\mathcal F_3$ should not contain elements for which $\Phi^K_1(P)$ and $\Phi^K_2(P)$ coincide.

Let us recall that, given a triangle, the trilinear coordinates $[t_1:t_2:t_3]$  of the points in a line $l$ satisfy an equation of the type
$$At_1+Bt_2+Ct_3=0 .$$

\noindent The author proves in \cite[Theorem 3]{K.FE}  that a line is a central line if and only if it has an equation $At_1+Bt_2+Ct_3=0 $ where $[A:B:C]$ are the coordinates of a triangle center. Moreover, there is a geometric relation between a triangle center and the central line corresponding to this same center function. Recall that given a triangle $ABC$ and a point $X$ not lying on the sides, the isogonal conjugate of $X$ with respect to $ABC$, denoted by $X'$, is constructed by reflecting the lines $XA$, $XB$, and $XC$ about the angle bisectors of $A$, $B$, and $C$ respectively. These three reflected lines concur at the isogonal conjugate of $X$. On the other hand the trilinear polar of $X$ is the axis of perspectivity of the cevian triangle of $X$ and the triangle $ABC$.

\begin{figure}[h!]
\begin{center}
\includegraphics[width=6cm]{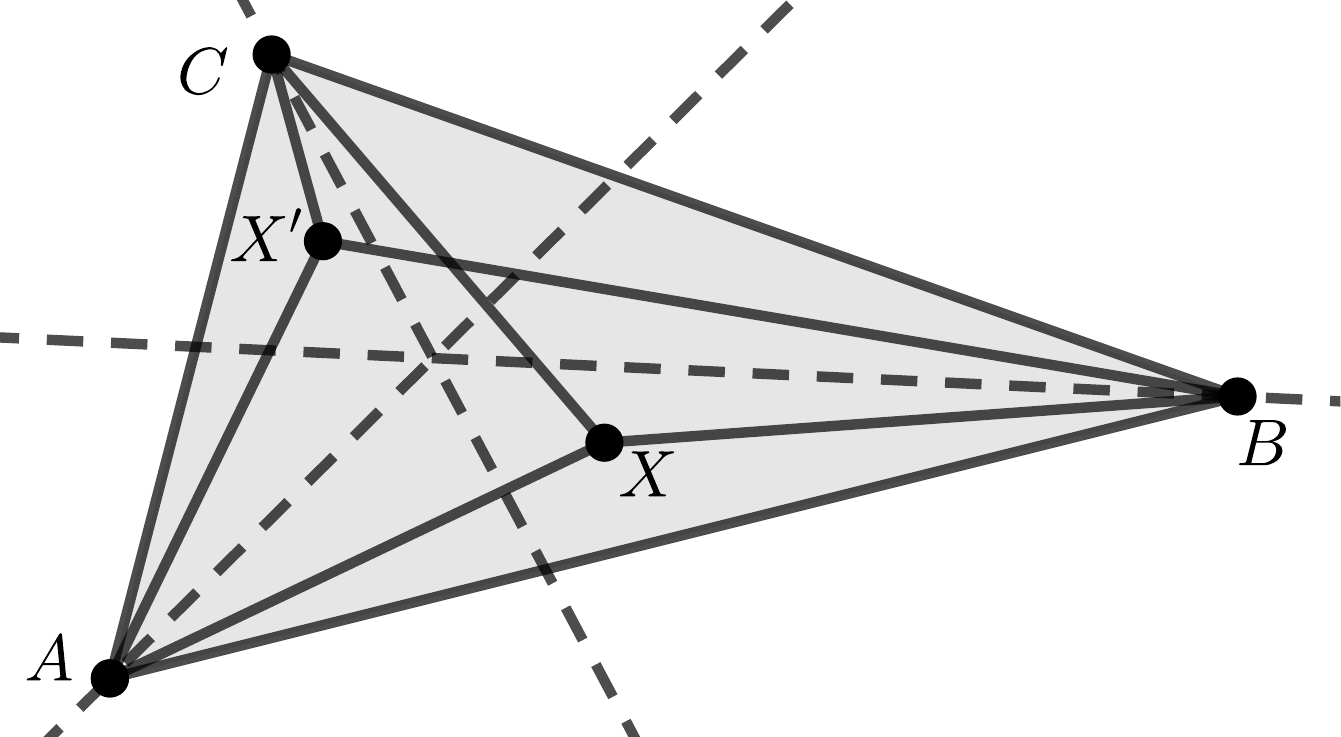}
\end{center}
\end{figure}

\begin{rem} Given a triangle center with trilinear coordinates $[x_0:y_0:z_0]$ not lying on the sides of the triangle (assume that none of the coordinates is 0) its \textbf{isogonal conjugate} is the center with trilinear coordinates $[\frac{1}{x_0}: \frac{1}{y_0}: \frac{1}{z_0}]$ and its \textbf{trilinear polar} is the line with equation
$$\frac{x}{x_0}+ \frac{y}{y_0}+ \frac{z}{z_0}=0 .$$

So the map that takes each triangle center $[x_0:y_0:z_0]$ to a central line $\{x_0x+y_0y+z_0z=0\}$ sends the point to the trilinear polar of its isogonal conjugate.

\end{rem}

\begin{rem} \label{theo.main3} 
Kimberling's approach to central lines (and also to central points) is close in spirit to Theorem \ref{theo.interpretation}. More precisely, suppose that we have a map $\Psi^K:\mathcal F_3\subset \mathcal P_3\to \mathcal L$. Then, the following three conditions are equivalent (see \cite{K.CL}):

\begin{itemize}

\item[(a)] $\Psi^K$ is a central line.

\item[(b)] If $P,P'$ are two similar triangles, then $\Psi^K(P)$ and $\Psi^K(P')$ are related via this same similarity.

\item[(c)] There exists a triangle center function $f$ such that for every triangle of reference, the equation of $\Psi^K(P)$ is of the type
$$f(a,b,c)t_1+f(b,c,a)t_2+f(c,a,b)t_3=0 $$

\noindent for some triangle center function $f$.

\end{itemize}

\end{rem}

\section{Some remarks concerning the geometry of affine combinations}

If we have $n$ points in $\mathbb R^2$,  we can identify each point in their span with a set of coefficients $(x_1,\ldots, x_n)$ such that $x_1+\ldots+x_n=1$, that is,
$$x_1V_1+\ldots+x_nV_n\sim(x_1,\ldots, x_n).$$

\noindent But this coefficients are not coordinates, in the sense that they are not unique.

\setcounter{figure}{1}
\begin{figure}[h!] 
\begin{center}
\includegraphics[width=4cm]{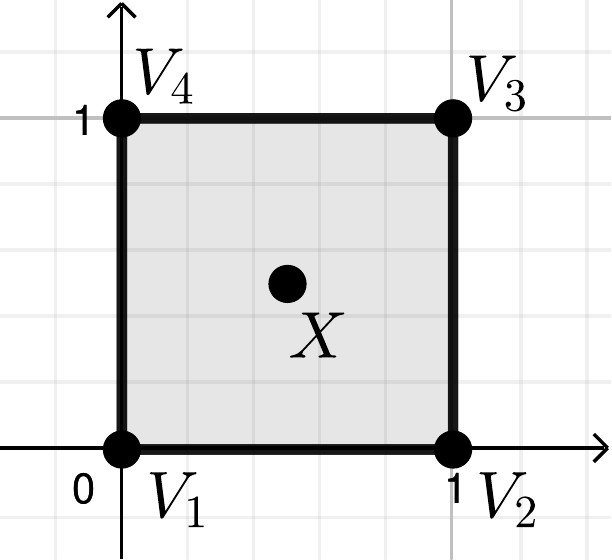}
\label{cuad}
\captionof{figure}{}
\end{center}
\end{figure}

\begin{exam} \label{exam.malditocuadrado} In Figure 2 
the point $X$ can be writen as
$$\begin{array}{l l} X&=\frac{1}{4}V_1+\frac{1}{4}V_2+\frac{1}{4}V_3+\frac{1}{4}V_4\\
\\
&=\frac{1}{6}V_1+\frac{1}{3}V_2+\frac{1}{6}V_3+\frac{1}{3}V_4.\end{array} $$

The expression $X=-V_1+\frac{1}{2}V_2+\frac{1}{2}V_4$ is not accepted in our setting since $-1+\frac{1}{2}+0+\frac{1}{2}=0$.

\end{exam}

The following lemma explains how we can describe lines using linear equations in the coefficients of the corresponding affine combination.

\begin{lemm} \label{lemm.n-2} Let $P\in\mathcal P_n$. Consider a system of $(n-1)$ independent compatible linear equations of the type
\begin{equation} \label{eq.n-2}\begin{bmatrix}a_{11} & \hdots & a_{1n} \\ \vdots & & \vdots \\ a_{n-2,1} & \hdots & a_{n-2,n} \\ 1 & \ldots & 1\end{bmatrix}\begin{bmatrix}x_1\\ \vdots \\ x_n\end{bmatrix}=\begin{bmatrix} 0 \\ \vdots \\ 0 \\ 1 \end{bmatrix}, \end{equation}

\noindent where $a_{ij}\in\mathbb{R}$. The set of points $X=x_1V_1+\ldots+x_nV_n $ such that $x_1,\ldots, x_n$ is a solution of the linear system is a line or a point.
\end{lemm}

\begin{proof} Consider a set of points $\widetilde V_1,\ldots,\widetilde V_n$ in $\mathbb R^n$ such that they are affinely independent and, for $i=1,\ldots, n$, the projection $\mathbb{R}^n\to\mathbb R^2$, $(x_1,x_2,\ldots, x_n)\mapsto (x_1,x_2)$, takes $\widetilde V_i$ to $V_i$.

The equations in the linear system described above correspond to a line in $\mathbb R^n$ if we interpret the solutions $x_1,\ldots, x_n$ as  barycentric coordinates (in this setting they are actually coordinates) with respect to the vertices $\widetilde V_1,\ldots, \widetilde V_n$. The projection of a line is either a line or a point.
\end{proof}

 \begin{rem} A point may be described by two different sets of coefficients, one of them satisfying a given system of linear equations and the other one not. 
 For example, in the setting of Example \ref{exam.malditocuadrado},  the coefficients $(\frac{1}{4},\frac{1}{4},\frac{1}{4},\frac{1}{4})$ satisfy the system of linear equations
\begin{equation}\label{eq.rectaejemplo}
\begin{cases}
x_1-x_2=0 ,\\ x_2-x_3=0 ,\\ x_1+x_2+x_3+x_4=1,
\end{cases} 
\end{equation}
but the coefficients $(\frac{1}{6},\frac{1}{3},\frac{1}{6},\frac{1}{3})$ do not and both sets of coefficients correspond to the same point $X$.
\end{rem}

Then the following problem arises: given a point in the plane and the equations of a line, how can we decide whether or not  the point belongs to the line? This question is answered with the following result.

\begin{lemm} Given a system of linear equations of the type \eqref{eq.n-2}, to decide whether or not the point
$$
X=\lambda_1V_1+\ldots+\lambda_nV_n \, , \text{ where }  \lambda_1+\ldots+\lambda_n=1, 
$$
lies on the corresponding line, we only need to check if the linear system obtained by adding the two linear equations 
\begin{equation}\label{eq.hay2}(x_1-\lambda_1)V_1+\ldots+(x_n-\lambda_n)V_n=0\end{equation}
\noindent has a solution.

\end{lemm}

To illustrate this lemma consider the tetragon in Example \ref{exam.malditocuadrado} and suppose that we want to decide if the point $X=\frac{1}{6}V_1+\frac{1}{3}V_2+\frac{1}{6}V_3+\frac{1}{3}V_4$ belongs to the line described by Equation \eqref{eq.rectaejemplo} (we already know that the answer is affirmative). We only need to check that the system
$$
\begin{bmatrix}1 & -1 & 0 & 0 \\ 0 & 1 & -1 & 0 \\ 0 & 1 & 1 & 0 \\ 0 & 0 & 1 & 1\\ 1 & 1 & 1 & 1 \end{bmatrix} \begin{bmatrix}x_1\\ x_2\\ x_3 \\ x_4 \end{bmatrix} =\begin{bmatrix}0\\ 0 \\ \frac{1}{2}\\ \frac{1}{2}\\ 1 \end{bmatrix} 
$$
has a solution, which is the case.
In this case the solution of the system is unique, but this may not happen. For instance, for the same tetragon and the same point, if we consider the line described by the system
$$
\begin{cases}
x_2-x_4=0\, , \\x_1-x_2-x_3+x_4=0\, , \\ x_1+x_2+x_3+x_4=1\, , 
\end{cases} 
$$

\noindent then the corresponding augmented system has infinitely many solutions.

\section{$n$-gon central lines} \label{sec:lines}

Following the idea in Definition \ref{defi} we can define many ``central objects''. For instance:

\begin{itemize}
    \item \textbf{central vertices}, like, for the set of scalene trigons with non collinear vertices, the vertex corresponding to the biggest (resp. to the smallest) angle   and
    
    \item  \textbf{central conics}, like the conic passing through the vertices, for the set of pentagons admitting one.
    
\end{itemize}

Central vertices (resp. central conics) are  functions that associate to every $(V_1,\ldots, V_n)$ a point in $\{V_1,\ldots, V_n\}$ (respectively a conic in the plane) and that commute with similarities and are invariant with respect to relabellings. Note that, in particular, central vertices are centers.

In this section we are mainly interested in central lines.

\begin{defi}  \label{defi.cl} Let $\Psi:\mathcal F_n\subset \mathcal P_n\to\mathcal L$. We say that $\Psi$ is a \textbf{central line} if (1) it commutes with similarities, that is, for every two similar $n$-gons $P,Q$, the corresponding lines $\Psi(P),\Psi(Q)$ are related via the same plane transformation and (2) it is invariant with respect to relabellings, that is, if two $n$-gons $P,Q$ differ only in their relabellings, then $\Psi(P)=\Psi(Q)$.
\end{defi}

\begin{rem} More ``central objects'' can be introduced  with the corresponding modifications. For example we can define
\begin{itemize}
    
    \item \textbf{central lengths}, like the perimeter,
    
    \item \textbf{central areas}, like, for the set of convex $n$-gons, the total area or the area of the $n$-gon obtained joining the midpoints of the sides,

    \item \textbf{central vector}, like, for the set of scalene trigons with their three vertices different, the vector pointing from the centroid to any central vertex.

\end{itemize}

Central lengths and central areas are functions that associate to every $(V_1,\ldots, V_n)$ a real number that is invariant with respect to relabellings and congruences. In addition, this function is required to be homogenoeus of degree 1 and 2, respectively.

A central vector is a function that associates to every $(V_1,\ldots, V_n)$ a vector in $\mathbb R^2$. It is invariant with respect to translations and commutes with every plane linear congruence.

\end{rem}

Our definition of central line is not a direct generalization of the definition given by Kimberling, so we  will use the following notion.

\begin{defi} Let $\Psi:\mathcal F_n\subset \mathcal P_n\to\mathcal L$. We say that $\Psi$ is a \textbf{Kimberling central line} if there exists two  centers $\Phi,\Phi'$ such that  for every $P\in\mathcal F_n$ the line $\Psi(P)$ passes through $\Phi(P)$ and $\Phi'(P)$ and $\Phi(P)$ and $\Phi'(P)$ do not coincide.

\end{defi}

Examples of central lines are the line containing the largest side of a $n$-gon (defined for the set of $n$-gons with a largest side) and the \emph{Schmidt line} for pentagons (see \cite{VT}). It is easy to see that the first example is not a Kimberling line.  A counterexample can be found for $n=3$ with any isosceles trigon. A proof is not possible yet since it would require  Theorem \ref{theo.symm}, but it relies on the fact that all the centers in such a triangle should lie on its symmetry axis (as we will see below).


An example of a central line for $n=4$ that is clearly a Kimberling central lines is the \emph{Euler line} (the line passing through the \emph{quasiorthocenter} and the \emph{quasicircumcenter}, as explained in \cite{M}).

Note that, according to Lemma \ref{lemma.d}, if a $n$-gon has a Kimberling central line, then all the elements in this line are centers. We will return to this question again in Section \ref{sect.symmetry}.

Unlike what happened for trigons, our concept of central line, the concept of Kimberling central line and a nice description in terms of equations are not always equivalent. We will clarify the relation between these concepts in the following theorem.

\begin{theo} \label{theo.main} Let $\Psi:\mathcal F_n\to\mathcal L$.  Consider the following statements:

\begin{itemize}

\item[(a)] $\Psi$ is a Kimberling central line.

\item[(b)] $\Psi$ is a central line.

\item[(c)] For $1\leq i\leq n-2$, $1\leq j\leq n$ there exist  functions $a_{ij}:\mathcal F_n\to\mathbb R$ such that for each $P=(V_1,\ldots,V_n)\in\mathcal F_n$, the system
$$\begin{bmatrix}1 & \ldots & 1 \\ a_{11}(P) & \ldots & a_{1n}(P) \\ \vdots & & \vdots \\ a_{n-2,1}(P) & \ldots & a_{n-2,n}(P)\end{bmatrix}\begin{bmatrix}x_1 \\ \vdots \\ x_n \end{bmatrix}=\begin{bmatrix} 1 \\ 0 \\ \vdots \\ 0 \end{bmatrix} $$

\begin{itemize}
    \item is compatible and the matrix of coefficients has rank $n-1$,
    \item the points in $\Psi(P)$  are of the type $X=x_1V_1+\ldots+x_nV_n$,
    \item for every $\alpha\in D_n$, and using the notation $\alpha(P)=(V_{\alpha(1)},\ldots,V_{\alpha(n)})$, we have
    $$\begin{bmatrix}1 & \ldots & 1 \\ a_{11}(P) & \ldots & a_{1n}(P) \\ \vdots & & \vdots \\ a_{n-2,1}(P) & \ldots & a_{n-2,n}(P)\end{bmatrix}\begin{bmatrix}x_1 \\ \vdots \\ x_n \end{bmatrix}=\begin{bmatrix} 1 \\ 0 \\ \vdots \\ 0 \end{bmatrix}\Longrightarrow \begin{bmatrix}1 & \ldots & 1 \\ a_{11}(\alpha (P)) & \ldots & a_{1n}(\alpha (P)) \\ \vdots & & \vdots \\ a_{n-2,1}(\alpha (P)) & \ldots & a_{n-2,n}(\alpha (P))\end{bmatrix}\begin{bmatrix}x_{\alpha(1)} \\ \vdots \\ x_{\alpha(n)} \end{bmatrix}=\begin{bmatrix} 1 \\ 0 \\ \vdots \\ 0 \end{bmatrix}$$
    \item for every similarity $T$ we have
    $$\begin{bmatrix}1 & \ldots & 1 \\ a_{11}(P) & \ldots & a_{1n}(P) \\ \vdots & & \vdots \\ a_{n-2,1}(P) & \ldots & a_{n-2,n}(P)\end{bmatrix}\begin{bmatrix}x_1 \\ \vdots \\ x_n \end{bmatrix}=\begin{bmatrix} 1 \\ 0 \\ \vdots \\ 0 \end{bmatrix}\Longrightarrow \begin{bmatrix}1 & \ldots & 1 \\ a_{11}(T( P)) & \ldots & a_{1n}(T(P)) \\ \vdots & & \vdots \\ a_{n-2,1}(T(P)) & \ldots & a_{n-2,n}(T(P))\end{bmatrix}\begin{bmatrix}x_1 \\ \vdots \\ x_n \end{bmatrix}=\begin{bmatrix} 1 \\ 0 \\ \vdots \\ 0 \end{bmatrix}$$

\end{itemize}

\end{itemize}

Statements (b) and (c) are equivalent and statement (a) implies both of them.

\end{theo}

The proof is immediate.

Statements (a) and (b) are not equivalent. The following is a good counterexample:

\begin{exam} \label{exam.rectangulo} Let $\mathcal F_4$ be the set of rectangles minus the set of squares. Let $\Psi:\mathcal F_4\to\mathcal L$ be the function that maps each element in $\mathcal F_4$ to the median crossing the largest sides. This is a central line in the sense of Definition \ref{defi.cl}. But for any element in $\mathcal F_4$ all the centers coincide and therefore $\Psi$ is not a Kimberling central line. 
To see this, note that $\mathcal F_4$ is a subset of the family  $\mathcal S_4$ defined in \cite{PS.C}. In Theorem 17 in that article it is proved that  for every element in $\mathcal S_4$ all the centers coincide.

\end{exam}

On the other hand, we have the following:

\begin{prop} \label{theo.main1}  Let $\mathcal F_n$ not containing elements with all their vertices equal and $\Psi:\mathcal F_n\to\mathcal L$.  If all the elements in $\mathcal F_n$ have a central vector parallel to $\Psi(P)$, then statement (b) (equiv. (c)) implies statement (a).

\end{prop}

\begin{proof} One of the two required centers is easy to find: take the point  $\Phi(P)$ in $\Psi(P)$ closest to the centroid. $\Phi(P)$ is a center.

Every $n$-gon has, at least, one central length: the perimeter $p(P)$. Define 
$$\Phi'(P)=\Phi(P)+p(P)v(P),$$

\noindent where $v(P)$ is the central vector ensured by the statement. $\Phi'(P)$ is a second center, different from $\Phi(P)$.

\end{proof}

From this theorem a new question arises. \emph{When does a $n$-gon have a central vector?} This will be considered in Section \ref{sect.symmetry}, but in most of the cases of $n$-gons with a central line we will ``have a central vector parallel to $\Psi(P)$'' as soon as $P$ is ``a little bit non-regular'' (this will be made precise later).

For most of the problems related to centers and central lines,  Kimberling central lines are a useful approach. So, in an analogue way to what Kimberling did, let us study the equations of a Kimberling central line.

\noindent \textbf{Method to obtain the equations of a Kimberling central line.}
Let $\Phi_1,\Phi_2:\mathcal F_n\to\mathbb R^2$ be two centers and $g_1,g_2:\widetilde{\mathcal F_n}\to\mathbb R$ be the two corresponding $n$-gon center fucntions. For a fixed $P=(V_1,\ldots, V_n)$ we write 
$$
a_{mk}:=g_m([d_{\rho^{k-1}(i)\rho^{k-1}(j)}]), 
\mbox{ where }  m=1,2 \mbox{ and }  1\leq k\leq n \, .
$$

If $X$ is in the line  through $Q_1(P)$ and $Q_2(P)$, then it admits and expression $X=\gamma_1V_1+\ldots+ \gamma_nV_n$ such that the matrix
$$
\begin{bmatrix}
a_{11} & \ldots & a_{1n} \\ a_{21} & \ldots & a_{2n}\\
\gamma_1 & \ldots & \gamma_n
\end{bmatrix}
$$
has rank $2$. 

Imposing rank $2$ for the above matrix we obtain $\binom{n}{3}$ equations (some of them redundant).  Note that the equation $\gamma_1+\ldots+\gamma_n=1$ is compatible with the equations obtained above since $a_{11}+ \cdots+ a_{1n}=1$ and $a_{21}+ \cdots+ a_{2n}=1$ and so $(1,\ldots,1)$ does not belong to the corresponding orthogonal space.

\begin{exam} \label{exam.simpleline4} Let us find the equations of the line passing through the centroid $\Phi_C$ and the simple center $\Phi_S$, whose expressions are
$$\Phi_C(P)=\frac{1}{4}V_1+\ldots+\frac{1}{4}V_4,\qquad \Phi_S(P)=\frac{1}{2(d_{13}+d_{24})}(d_{13}V_1+d_{24}V_2+d_{13}V_3+d_{24}V_4) .$$
These are defined for the family of $n$-gons such that these centers do not coincide.

To find the equations of the line through $\Phi_C(P)$ and $\Phi_S(P)$ we impose rank 2 for the matrix
$$
\begin{bmatrix}
1 & 1 & 1 & 1 \\ d_{13} & d_{24} & d_{13} & d_{24}\\
\gamma_1 & \gamma_2 & \gamma_3 & \gamma_4
\end{bmatrix}.
$$
If $A_i$ denotes the $3\times 3$ matrix with columns $i,i+1,i+2 \mod 4$, then we have
\begin{eqnarray*}
det(A_1)=-det(A_3)&=&(d_{13}-d_{24})\gamma_1+(d_{24}-d_{13})\gamma_3,\\ 
det(A_2)=-det(A_4)&=&(d_{24}-d_{13})\gamma_2+(d_{13}-d_{24})\gamma_4,
\end{eqnarray*}
and therefore we obtain the equations 
$$
\gamma_1-\gamma_3=0, \, \mbox{ and } \,
\gamma_2-\gamma_4=0 .
$$
Therefore the equations for the line are
$$\begin{bmatrix}1 & 1 & 1 & 1 \\ 1 & 0 & -1 & 0 \\ 0 & 1 & 0 & -1\end{bmatrix} \begin{bmatrix}\gamma_1\\ \gamma_2\\ \gamma_3 \\ \gamma_4\end{bmatrix}=\begin{bmatrix} 1 \\ 0 \\ 0 \end{bmatrix} .$$

\end{exam}

\section{Symmetries, coincidence and collinearity of centers} \label{sect.symmetry}

We say that the group of symmetries of $P\in\mathcal P_n$ is the group of plane rigid motions that fix the set of vertices.

\begin{theo} \label{theo.symm} Let $\Phi:\mathcal F_n\to\mathbb R^2$ be some $n$-gon center. Then, for every $n$-gon $P\in\mathcal F_n$, $\Phi(P)$ belongs to the set of fixed points of the symmetry group of $P$.

\end{theo}

\begin{proof}  For every rigid motion $T$ in the symmetry group of $P$, we have that the action of this rigid motion on $P$ can be viewed as a relabelling $\alpha(P)$ for some $\alpha\in D_n$. So $T(\Phi(P))=\Phi(T(P))=\Phi(\alpha (P))=\Phi(P)$.
\end{proof}

Coincidence and collinearity of centers and the existence of a central vector are, in some sense, signs of symmetry in the $n$-gon. See \cites{AS.Q, PS.C} and for a different but similar context \cite{E}.

For the triangle, this kind of problems have been very well studied (although this has not been stated in these terms). We have the following:

\begin{theo} \label{theo.symmetryiff} Let $P$ be a trigon.
\begin{itemize}
    \item[(a)] All the centers in $P$ are collinear if and only if $P$ is isosceles.
    
    \item[(b)] All the centers in $P$ are coincident if and only if $P$ is equilateral.
    
\end{itemize}

\end{theo}

\begin{proof}

\begin{itemize}
    \item[(a)] If a trigon is isosceles, then all its centers are collinear, since $P$ has an axis of symmetry. On the other hand, we know that the centroid, the circumcenter and the incenter are collinear if and only if the trigon is isosceles (the incenter of a trigon lies in the Euler line if and only if it is isosceles, see \cite{F}, for instance).

    \item[(b)] If $P=(V_1,V_2,V_3)$ is equilateral, then all its centers are coincident, since $P$ has three concurrent and different axis of symmetry. On the other hand, the incenter of a trigon $P$ is
    $$X=\frac{d_{23}}{p}V_1+\frac{d_{13}}{p}V_2+\frac{d_{24}}{p}V_3, $$
    
    \noindent where $p$ denotes the perimeter of $P$. So we can see that this point coincides with the centroid if an only if $P$ is equilateral.

\end{itemize}

\end{proof}

 In the case $n\geq 4$, we have not guaranteed the existence of three fixed centers (functions) that are collinear if and only if the set of fixed points of the symmetry group consists of a single axis (as happens for the centroid, circumcenter and incenter for $n=3$), neither of two fixed centers (functions) that are  coincident if and only if the set of fixed points of the symmetry group consists of a single point (as happens for the centroid and the incenter for $n=3$).

Moreover, we cannot even guarantee the existence of three non-collinear centers (points) for each $P$ with no axis of symmetry, neither of two non-coincident centers (points) for $n$-gons such that the set of fixed points of the symmetry group consists of more than a single point.

The strongest statement that we can make so far is the following:

\begin{theo} \label{theo.symmetry} Let $P$ be a $n$-gon such that not all its vertices are equal.

\begin{itemize}
    \item[(a)] $P$ has a central vector $v(P)$ if and only if $P$ has two centers which are not coincident. Every point on the line determined by the centroid and $v(P)$ is the (image of) some center.
    
    \item[(b)] $P$ has two linearly independent central vectors $v_1(P),v_2(P)$ if and only if $P$ has three centers which are not collinear. Every point in the plane is the (image of) some center.
\end{itemize}

\end{theo}

\begin{proof} We start with the first statement of each part. Let us denote by $p(P)$ the perimeter of a $n$-gon $P\in\mathcal P_n$.

\begin{itemize}
    \item[(a)] From a central vector $v(P)$ we can always obtain two centers: the centroid $\Phi_C(P)$ and $\Phi_C(P)+p(P)v(P)$. Conversely, from two centers $\Phi_1,\Phi_2$ we can obtain a central vector taking the one from the point $\Phi_1(P)$ to $\Phi_2(P)$.
    
    \item[(b)] From two such vectors, we can obtain three centers: the centroid $\Phi_C(P)$, $\Phi_C(P)+p(P)v_1(P)$ and $\Phi_C(P)+p(P)v_2(P)$. Conversely, from such three centers $\Phi_1,\Phi_2,\Phi_3$, we can obtain two central vectors taking the one from  $\Phi_1(P)$ to $\Phi_2(P)$ and the one from $\Phi_1(P)$ to $\Phi_3(P)$.

\end{itemize}

The rest is a consequence of  Lemma \ref{lemma.d}.

\end{proof}

So far we have used the conditions  ``if a $n$-gon has two different centers'', ``if a $n$-gon has three centers which are affinely independent'' and ``if a $n$-gon has a central vector in a given direction''. Now we will clarify them, including a list of sufficient conditions that imply them.

\begin{theo}  If a $n$-gon satsifies any of the properties

\begin{enumerate}

    \item One of the sides is greater (smaller) than the others,

    \item It is convex and the interior angles between two adjacent sides is greater (smaller) than the others,

\end{enumerate}

\noindent then it has a central vector and thus has two different centers.

\end{theo}

\begin{proof} In each of the cases:

\begin{enumerate}
    \item Consider the vector joining the centroid with the midpoint of the greatest (smallest) side.
    
    \item Consider the vector joining the centroid with the vertex corresponding to the largest (smallest) angle.
    
\end{enumerate}

\end{proof}

\begin{theo} If a $n$-gon has any of the properties

\begin{enumerate}
    \item It is convex and scalene,
    
    \item It is convex and all the angles are different,
\end{enumerate}

\noindent then it has two linearly independent central vectors and thus has three different centers.

\end{theo}

\begin{proof} In each of the cases:

\begin{enumerate}
    \item Consider the vector joining the centroid and the midpoint of the largest side, and the vector joining the centroid and the midpoint of the smallest side. If both of them are collinear, then consider the vector joining the centroid  and the midpoint of the second largest side.
    
    \item Consider the vector joining the centroid and the vertex corresponding to the  largest angle, and the vector joining the centroid and the vertex corresponding to the  smallest angle. If both of them are collinear, then consider the vector joining the centroid  and the vertex corresponding to the second largest angle.

\end{enumerate}

\end{proof}

Finally, note that if we have a central vector and a central direction which are not perpendicular, the projection of a central vector in this central direction is another central vector.

\section{Two applications} \label{section.app}

\textbf{Tangential $n$-gons} are those convex $n$-gons admitting an inscribed circle. Let us denote by $\mathcal I_n$ the set of tangential $n$-gons. For this set we can define the \textbf{incenter} $I$ as the center of this inscribed circle.

\begin{figure}[h!]
\centering
\includegraphics[width=6cm]{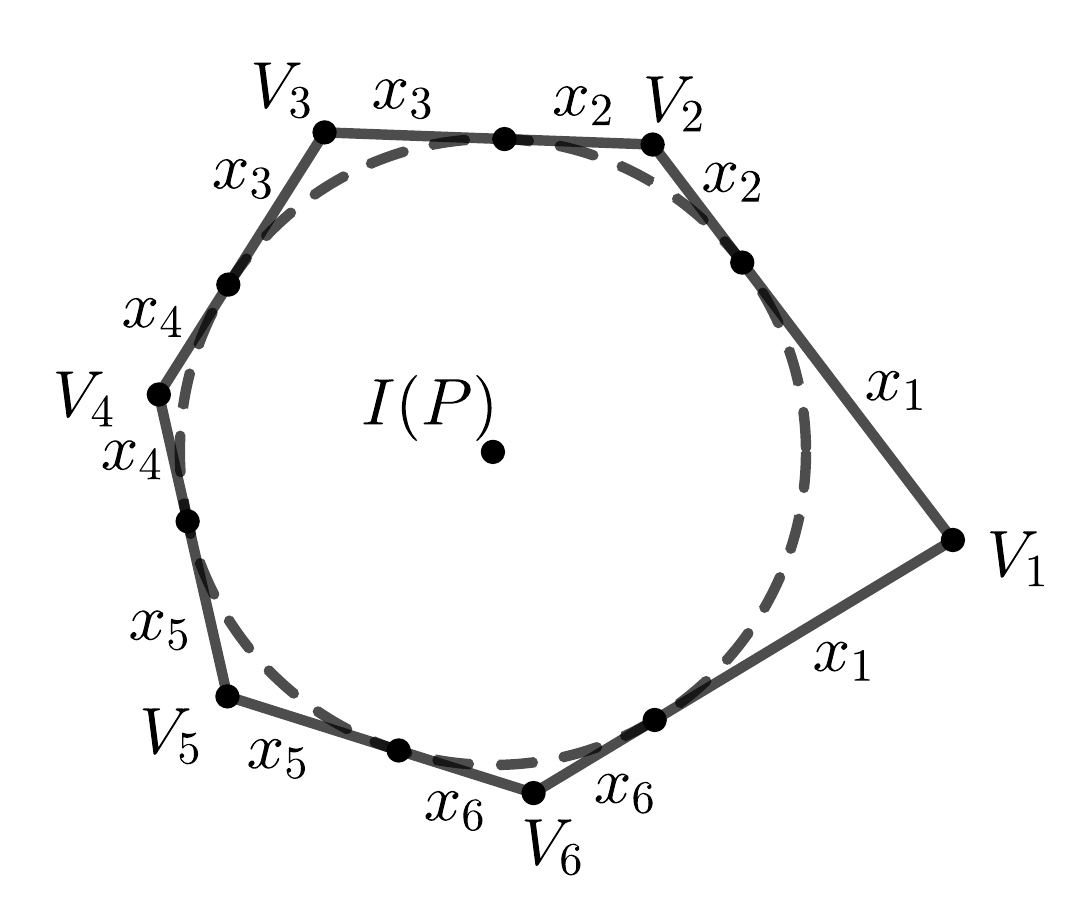}
\caption{Picture corresponding to a tangential hexagon.}
\end{figure}

An elementary argument shows that for every set of side lengths $d_{12},\ldots, d_{n1}$ such that the system of equations
$$\begin{cases}
d_{12}=x_1+x_2\, ,\\ \qquad \vdots \\ d_{n1}=x_n+x_1 \, ,
\end{cases}$$

\noindent has a solution $(x_1,\ldots,x_n)$ consisting of real postitive numbers, there is a tangential convex $n$-gon with these side lengths. These $x_i$ are the tangent lengths of the $n$-gons, that is, the distance from $V_i$ to the points where the circle is tangent to the sides $V_iV_{i-1}$ and $V_iV_{i+1}$.


The incenter is a center in the sense described in this article. There is a description in \cite{M} of the incenter for $n=3$ and $n=4$. Our approach allows us to generalize this description for $n\geq 4$.

\begin{prop} The  center function
$$
g_I([d_{ij}])=x_n([d_{ij}])+x_2([d_{ij}])
$$
corresponds to the incenter.

\end{prop}

\begin{proof} We can assume that the result is true for $n\leq 4$.

We proceed in a similar way as in \cite{M}. We may suppose that the sum of two adjacent angles in $P$ is greater than $\pi$. Otherwise the sum of the angles would be less than or equal to $n\pi/2$, which is a contradiction, for $n>5$, with the fact that the sum of the angles is $(n-2)\pi$.

Let us assume that these angles correspond to $V_1,V_{n+1}$. Consider the $n$-gon $P'=(V_0,V_2,\ldots, V_{n-1})$, where $V_0$ is the point where $V_1V_2$ and $V_nV_{n+1}$ meet.

\begin{figure}[!ht]
\centering
\includegraphics[width=4.5cm]{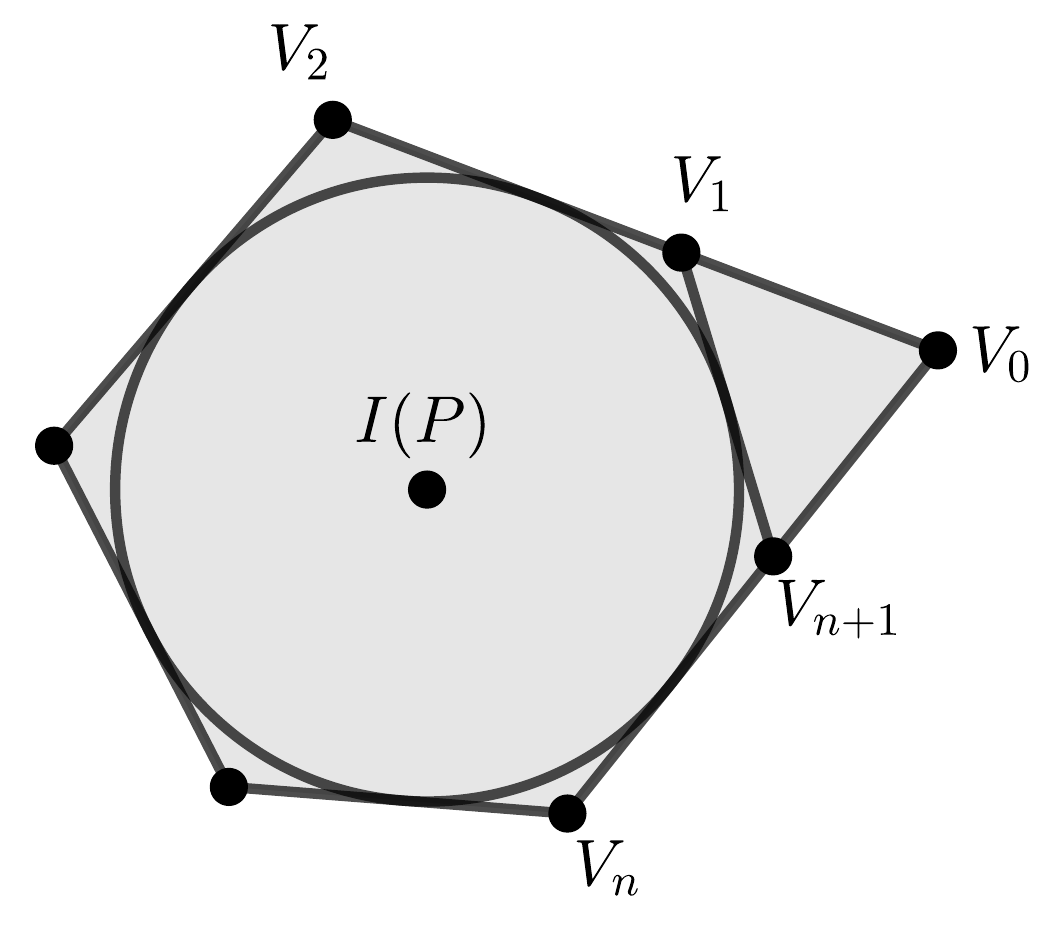}
\end{figure}

$\Phi_I(P)$ is the incenter of $P'$. Let us denote by $p'$ the perimeter of $P'$. Then, by hypothesis,
\begin{equation}\label{eq.incentrogorro} p'\Phi_I(P)=[(x_2+x_n)V_0+(x_0+x_3)V_2+\ldots+(x_{n-1}+x_0)V_n] \, .
\end{equation}

On the other hand, $\Phi_I(P)$ is the excenter relative to the vertex $V_0$ of the trigon $(V_0,V_1,V_{n+1})$ and therefore we have
\begin{equation} \label{eq.excenter}
2(x_0-x_1-x_{n+1})\Phi_I(P)=[-(x_1+x_{n+1})V_0+(x_0-x_{n+1})V_1+(x_0-x_1)V_{n+1}]\, .
\end{equation}

Note that, since $V_0,V_1,V_2$ and $V_0,V_{n+1},V_n$ are collinear we have that
\begin{equation}\label{eq.collinear}
(0,0)=(x_1+x_2+x_n+x_{n+1})V_0+(-x_0-x_2)V_1+(x_0-x_1)V_2+(x_0-x_{n+1})V_n+(-x_0-x_n)V_{n+1}\, .
\end{equation}

Substracting Equations \eqref{eq.collinear} and \eqref{eq.excenter} from Equation \eqref{eq.incentrogorro} we obtain the desired result.

\end{proof}


In  \cite[Corollary 7]{AM} the authors proved the following ``simple but surprising relation'':
\begin{theo} For any $P\in\mathcal I_n$, the incenter, the center of mass of the boundary and the center of mass of the lamina (see \cite[Example 13]{PS.C} additionally) of any $n$-gon are collinear.
\end{theo}

Thanks to our approach we can look, artificially, with no geometric intuition, for theorems similar to this one. Let us provide an example in the following result, inspired in a simple computation. We do not claim that it is very interesting geometrically, but it could have been difficult to identify by geometric means.

\begin{theo}  For any parallelogram, the centroid of the boundary belongs to the line passing through the centroid and the simple center.

\end{theo}

\begin{proof}  According to Example \ref{exam.simpleline4}, the equations of this line are
$$\begin{bmatrix}1 & 1 & 1 & 1 \\  1 & 0 & -1 & 0 \\0 & 1 & 0 & -1\end{bmatrix} \begin{bmatrix}\gamma_1\\ \gamma_2\\ \gamma_3\\ \gamma_4\end{bmatrix} =\begin{bmatrix}1 \\ 0 \\ 0 \end{bmatrix} \, .$$

Since the centroid of the boundary is
$$\Phi_B(P)=\frac{d_{14}+d_{12}}{2p}V_1+\ldots+\frac{d_{43}+d_{41}}{2p}V_4 \, ,$$

\noindent where $p$ denotes the perimeter, we can see that the coefficients satisfy the equations if and only if
$$d_{14}+d_{12}=d_{32}+d_{34}\qquad \mbox{and} \qquad d_{12}+d_{32}=d_{43}+d_{41} \, ,$$

\noindent which is guaranteed for parallelograms.

\end{proof}

\section{Final Remarks}

In this section we will discuss some questions that remain open and possible lines for future work.

\subsection{Problems of coincidence and collinearity}

The problem of deciding if two centers coincide (resp. three centers are collinear) for a fixed $n$-gon $P$ is easy to solve. A very interesting problem would be the following. 

\noindent  \textbf{Problem of Coincidence:} \emph{Given two $n$-gon center functions $g_1,g_2$, determine the set}
$$ \mathcal G_{g_1,g_2}=\{P\in\mathcal P_n:\Phi_{g_1}(P)=\Phi_{g_2}(P)\} .$$

\noindent  \textbf{Problem of Collinearity:} \emph{Given three $n$-gon center functions $g_1,g_2,g_3$, determine the set}
$$ \mathcal G_{g_1,g_2,g_3}=\{P\in\mathcal P_n:\Phi_{g_1}(P),\Phi_{g_2}(P),\Phi_{g_3}(P)\text{ are collinear}\}  .$$

\medskip

Regarding the first problem note that, in general, the elements satisfying that 
\begin{equation}\label{eq.final}(g_1([d_{ij}]),\ldots,g_1([d_{\rho^{n-1}(i)\rho^{n-1}(j)})]) \mbox{ and } (g_2([d_{ij}]),\ldots,g_2([d_{\rho^{n-1}(i)\rho^{n-1}(j)}]))\; \mbox{are proportional}\end{equation}

\noindent  are in $G_{g_1,g_2}$ but the converse is not true.

\begin{exam} \label{exam.CG} For $n=4$, consider the tetragon center functions $g_1,g_2:\widetilde{\mathcal P}_4\to\mathbb R$ given by $g_1([d_{ij}])=1$ and $g_2([d_{ij}])=d_{13}$ (corresponding to the centroid and the simple center). The set of tetragons satisfying Equation \eqref{eq.final} is the set of tetragons with equal diagonals. But any rhombus with different diagonals is also in $\mathcal G_{g_1,g_2}$.

\end{exam}

\subsection{Polygons satisfying a non-linear cyclic equation}

The equations $\Phi_{g_1}(P)=\Phi_{g_2}(P)$ are, under some manipulations, of the type
$$h(d_{ij})V_1+\ldots+h(d_{\rho^{n-1}(i)\rho^{n-1}(j)})V_n=0,\\$$ 
where $h$ is ``almost'' (with the possible exception of the homogeneity property) a center function  and $h(d_{ij})+\ldots+h(d_{\rho^{n-1}(i)\rho^{n-1}(j)})=0$. An algebraic study of these equations would provide more tools to solve Problems of Coincidence. 
Since these equations share some properties with those studied in \cite{BS}, it seems to us that this book could be a good starting point.

\subsection{Centers of finite sets of points} \label{comment3}

As we already explained in the introduction, in \cite{E} there is a definition of center for $n$-simplices in $\mathbb R^d$. This definition can be naturally extended for sets of $n$-points $\{V_1,\ldots,V_n\}$  (instead of for $n$-tuples $(V_1,\ldots,V_n)$) in $\mathbb R^d$. It is very similar to Definition \ref{defi}.

This definition for sets of points does not take into account any relation of adjacency between vertices. In $\mathbb R^2$, the notion of $n$-gon considers a very special and simple type of adjacency (all the vertices are ``equal'': they connect two edges). But, for instance, the ways to connect a set of points is much more complicated. Also it is possible to provide a definition of central line in this setting.

For instance, the centroid (as we have defined it here) is a center that does not take into account any adjacency relation. Many problems in Applied Mathematics require this other general approach. For example, in the study of tumor growth (see \cite{J}) it is important to determine the centroid of the tumor (maybe other centers, apart from the centroid, could be relevant). Also, in Computational Chemistry central lines could be useful in problems concerning porosity. In this kind of problems, the target is to find trajectories to put an atom inside a ``cage'' (a molecule whose vertices are other atoms) in order to change the properties of the material (see \cite{G}).

\subsection{Symmetries, coincidences and collinearity}

As explained in \cites{AS.Q, PS.C}, and also in this paper, there is a strong relation between symmetries of the $n$-gon and coincidence and collinearity of centers.

It would be desirable to find a characterization (not only a list of sufficient conditions) for a $n$-gon to have two non-coincident centers (equiv. a central vector) or three non-collinear centers (equiv. two linearly independent central vectors).

Moreover, there are two very important questions: 

\begin{enumerate}
    \item Can we find two  centers $\Phi_1,\Phi_2$, defined for every $P\in\mathcal P_n$, playing the role of the incenter and the centroid for $n=3$, that is, the set of fixed points of the symmetry group of $P$ consists of a single point if and only if $\Phi_1(P)=\Phi_2(P)$?

    \item Can we find three  centers $\Phi_1,\Phi_2,\Phi_3$, defined for every $P\in\mathcal P_n$, playing the role of the incenter, the centroid and the circumcenter for $n=3$, that is, the set of fixed points of the symmetry group of $P$ consists of a line if and only if $\Phi_1(P)$, $\Phi_2(P)$ and $\Phi_3(P)$ are collinear?

\end{enumerate}

In other words, we should look for three centers whose relative position determines the symmetry group of $P$.

\subsection{The coefficients of the equations of a central line}

The coefficients of the equation of a central line for the case $n=3$ have a nice interpretation: they are the coefficients of a center (see Remark \ref{theo.main3}) . Is there any interpretation of the coefficients of the linear system of equations of a central line for $n\geq 4$ (see (c) in Theorem \ref{theo.main})?

\subsection{Other central objects}

We have already mentioned the possibility of defining other ``central objects'' apart from points and lines, such as vectors and lengths.
A natural question is to define central curves. Due to their relation to problems in plane geometry, it would be desirable to define ``central conics'', for a start. 
They would also be of great importance in the context of Section \ref{comment3} (problems of approximation).

\subsection{Centers of Jordan curves}

A natural generalization of the concepts of center and central line for $n$-gons are the concepts of center and central line for Jordan curves. 
There exist many examples in the bibliography concerning points called ``centers'' for Jordan curves, such as the equichordal points or the power points (see \cite{P} and the references therein).
In this case, the expression in part (b) of Theorem \ref{theo.interpretation} would translate into an expression involving integrals instead of sums.



\begin{thebibliography}{}
%
%










\bibitem{AS.Q} A. Al-Sharif, M. Hajja and P. T. Krasopoulos, \emph{Coincidences of Centers of Plane Quadrilaterals}, Results in Mathematics 55.3,  231--247 (2009).





\bibitem{AM} T. M. Apostol, M. A.  Mnatsakanian, \emph{Figures Circunscribing Circles}, The American Mathematical Monthly 111.10, 853--863 (2004). 



\bibitem{BS} F. Bachmann, E. Schmidt, \emph{N-gons}, University of Toronto Press (2016).






\bibitem{BB} J. C. Bowers, P. L. Bowers, \emph{A menger redux: embedding metric spaces isometrically in Euclidean space}, The American Mathematical Monthly, 124(7), 621-636 (2017).



\bibitem{E} A. L. Edmonds, \emph{The center conjecture for equifacetal simplices}, Advances in Geometry 9.4 (2009).


\bibitem{F} W. N. Franzsen, \emph{The distance from the incenter to the Euler line}, Forum Geometricorum 11, 231--236 (2011).




\bibitem{G} I. G. Garc\'ia, M. Bernabei, M. Haranczyk,  \emph{Toward Automated Tools for Characterization of Molecular Porosity}, Journal of Chemical Theory and Computation 15.1, 787--798 (2018).


\bibitem{J} J. Jim\'enez-S\'anchez et al., \emph{Evolutionary dynamics at the tumor edge reveal metabolic imaging biomarkers}, Proceedings of the National Academy of Sciences, 118.6 (2021).













\bibitem{K.FE}  C. Kimberling, \emph{Functional equations associated with triangle geometry}, Aequationes Math. 45, 127--152  (1993).

\bibitem{K.CL} C. Kimberling, \emph{Central Points and Central Lines in the Plane of a Triangle}, Mathematics Magazine. 67.3, 163--187 (1994).


\bibitem{K.E} C. Kimberling, Encyclopedia of Triangle centers, \url{https://faculty.evansville.edu/ck6/encyclopedia/etc.html}.



\bibitem{LL} L. Liberti, C. Lavor, \emph{Euclidean Distance Geometry}, Springer, Berlin (2017).










\bibitem{M} A. Myakishev, \emph{On Two Remarkable Lines Related to a Quadrilateral}, Forum Geometricorum 6, 289--295 (2006).




\bibitem{P} L. F. Prieto-Mart\'inez, \emph{Geometric continuity in terms of Riordan matrices and the F-chordal Problem}, Revista de la Real Academia de Ciencas Exactas, F\'isicas y Naturales, Seria A, Matem\'aticas, 115 (2021).


\bibitem{PS.C} L. F. Prieto-Mart\'inez, R. S\'anchez-Cauce, \emph{Generalization of Kimberling's concept of triangle center for other polygons}, Results in Mathematics 76.2, 1--18 (2021).








\bibitem{S.NDIM} D. M. L. Y. Sommerville, \emph{Introduction to the Geometry of $N$ Dimensions},  Courier Dover Publications, New York  (2020).




\bibitem{VT} C. Van Tienhoven, \emph{Encyclopedia of Quadri-Figures}, \url{https://chrisvantienhoven.nl/}.








\end{thebibliography}


\end{document}